\documentclass[12pt,a4paper]{amsart}

\usepackage[english]{babel}
\usepackage{amssymb,amsthm,amsmath}

\newtheorem{theorem}{Theorem}[section]
\newtheorem{lemma}[theorem]{Lemma}
\newtheorem{corollary}[theorem]{Corollary}
\newtheorem{proposition}[theorem]{Proposition}

\theoremstyle{remark}
\newtheorem{definition}[theorem]{Definition}
\newtheorem{remark}[theorem]{Remark}
\newtheorem{example}[theorem]{Example}

\newtheorem{conjecture}[theorem]{Conjecture}

\numberwithin{equation}{section}

\thanks{
The author was supported in part by JSPS Grant-in-Aid for 
Scientific Research (B) 22340025.
}

\dedicatory{Dedicated to Professor Matti Vuorinen \\
on the occasion of his sixty-fifth birthday}

\newcommand{\C}{{\mathbb C}}
\newcommand{\D}{{\mathbb D}}

\newcommand{\Cd}{{\widetilde C}}
\newcommand{\Cdm}{{\widetilde C'}}
\newcommand{\Ce}{{\widehat C}}
\newcommand{\Cem}{{\widehat C'}}

\newcommand{\sphere}{{\widehat{\mathbb C}}}

\newcommand{\isom}{{\operatorname{Isom}}}

\newcommand{\mob}{{\text{\rm M\"{o}b}}}

\newcommand{\inv}{^{-1}}

\newcommand{\PSL}{{\operatorname{PSL}}}
\newcommand{\SL}{{\operatorname{SL}}}

\newcommand{\aand}{{\quad\text{and}\quad}}

\begin{document}
\bibliographystyle{amsplain}

\title[Hyperbolic metric and uniform perfectness] 
{\vspace*{.1cm}
Spherical density of hyperbolic metric \\
and uniform perfectness
}


\author[T. Sugawa]{\noindent Toshiyuki Sugawa}
\email{sugawa@math.is.tohoku.ac.jp}
\address{\newline Graduate School of Information Sciences
\newline Tohoku University
\newline Sendai 980-8579
\newline Japan
}


\begin{abstract} 
It is well known that a hyperbolic domain in the complex plane has
uniformly perfect boundary precisely when the product of its hyperbolic
density and the distance function to its boundary has a positive lower bound.
We extend this characterization to a hyperbolic domain in the Riemann sphere
in terms of the spherical metric.
\end{abstract}

\maketitle

\section{Introduction and main result}
Let $\Omega$ be a domain in the Riemann sphere $\sphere=\C\cup\{\infty\}$
with at least three points in its boundary $\partial\Omega\subset\sphere.$
Then, it is well known that $\Omega$ carries the {\it hyperbolic} metric
$\lambda_\Omega=\lambda_\Omega(z)|dz|,$ which is a complete conformal
metric of constant Gaussian curvature $-4.$
Such a domain is thus called hyperbolic.
For instance, the unit disk $\D=\{z\in\C: |z|<1\}$ has the hyperbolic
metric of the form
$$
\lambda_\D(z)=\frac1{1-|z|^2}.
$$
In what follows, we consider only hyperbolic domains unless otherwise stated.
The hyperbolic metric $\lambda_\Omega$
can be characterized by the relation
$$
\lambda_\D(z)
=\lambda_\Omega(p(z))|p'(z)|,\quad
z\in\Omega,
$$
where $p:\D\to\Omega$ is an analytic universal coverning projection.

As general references for the hyperbolic metric and related topics, 
the reader may consult \cite{KL:hg}, \cite{AW:sp}, and \cite{BM05}.
We remark that the hyperbolic metric often refers to $2\lambda_\Omega,$
which is of constant curvature $-1.$
The reader should check its definition first
when refering to other papers or books on the hyperbolic metric.

We denote by $d_\Omega(z)$ the Euclidean distance from $z\in\Omega$
to the boundary $\partial\Omega;$ namely,
$$
d_\Omega(z)=\min_{a\in\partial\Omega}|z-a|.
$$
As is easily seen, the inequality $d_\Omega(z)\lambda_\Omega(z)\le1$
holds for each $z\in\Omega\setminus\{\infty\}.$
Moreover, if $\Omega$ is simply connected and if $\Omega\subset\C,$
the Koebe one-quarter theorem implies the opposite inequality
$d_\Omega(z)\lambda_\Omega(z)\ge1/4.$
In general, however, $d_\Omega(z)\lambda_\Omega(z)$ can be arbitrarily small.
Indeed, positivity of the quantity
$$
C(\Omega)=\inf_{z\in\Omega}d_\Omega(z)\lambda_\Omega(z)
$$
gives the domain $\Omega$ a strong geometric constraint.

\begin{theorem}[Beardon and Pommerenke \cite{BP78}]
Let $\Omega$ be a hyperbolic domain in $\C.$
Then $C(\Omega)>0$ if and only if $\partial\Omega$ is uniformly perfect.
\end{theorem}

Here, a compact subset $E$ of $\sphere$ containing at least two points is
said to be {\it uniformly perfect} if there exists a constant $k\in(0,1)$
such that $\{z\in E: kr<|z-a|<r\}\ne\emptyset$ for every 
$a\in E\setminus\{\infty\}$ and $0<r<d(E),$
where $d(E)$ denotes the Euclidean diameter of $E.$
Note that $d(E)=+\infty$ whenever $\infty\in E.$
There are many other characterizations of uniformly perfect sets.
See \cite{Pom79}, \cite{Pom84}, \cite{SugawaUP} and \cite{SugawaSEUP}
in addition to \cite{KL:hg} and \cite{AW:sp}.

In the above theorem, the assumption $\Omega\subset\C$ is essential.
Indeed, let us consider the domain $\Delta_R=\{z\in\sphere: |z|>R\}$
containing $\infty.$
Then, the hyperbolic metric of it is expressed by
$$
\lambda_{\Delta_R}(z)=\frac{R}{|z|^2-R^2}.
$$
Thus,
$$
d_{\Delta_R}(z)\lambda_{\Delta_R}(z)=\frac{R}{|z|+R}\to0 \quad (z\to\infty).
$$
This phenomenon may be explained by the fact that
$\Delta_R$ and $\Delta_R\setminus\{\infty\}$ cannot be distinguished 
merely by the distance function $d_\Omega(z).$

It is therefore desirable to have a similar characterization of
the uniform perfectness which is valid  for domains in $\sphere.$
To this end, it is natural to employ the spherical distance instead of
the Euclidean one.

We recall that the spherical (chordal) distance is defined by
$$
\sigma(z,w)=\frac{|z-w|}{\sqrt{(1+|z|^2)(1+|w|^2)}}
$$
for $z, w\in\C$ and $\sigma(z,\infty)=1/\sqrt{1+|z|^2}$ for $z\in\C.$
Note that $0\le\sigma(z,w)\le 1.$
The corresponding infinitesimal form is given by
$$
\sigma(z)|dz|=\frac{|dz|}{1+|z|^2}
$$
which is known as the spherical metric and has constant Gauassian curvature
$+4.$
It is also convenient to use the quantity
$$
\tau(z,w)=\left|\frac{z-w}{1+z\bar w}\right|,
$$
which can also be thought of as a spherical counterpart of the Euclidean
distance, although $\tau$ is not a distance function on $\sphere.$
We then consider the distances to the boundary
$$
\delta_\Omega(z)=\min_{a\in\partial\Omega}\sigma(z,a)
\aand
\varepsilon_\Omega(z)=\min_{a\in\partial\Omega}\tau(z,a)
$$
for $z\in\Omega.$

In the context of spherical geometry, it is more natural to
consider the {\it spherical density} of the hyperbolic metric defined by
$$
\mu_\Omega(z)=\frac{\lambda_\Omega(z)|dz|}{\sigma(z)|dz|}
=(1+|z|^2)\lambda_\Omega(z).
$$
Minda \cite{Minda85a} studied $\mu_\Omega(z)$ in relation with 
$\varepsilon_\Omega(z)$ and gave several estimates for $\mu_\Omega(z).$
Among others, the following result is relevant to the present paper.

\begin{theorem}[Minda \cite{Minda85a}]
Let $\Omega$ be a hyperbolic domain in $\sphere.$
For each $z\in\Omega,$ the inequality $\varepsilon_\Omega(z)\mu_\Omega(z)
\le 1.$ Moreover, equality holds at $z$ if and only if $\Omega$ is a spherical
disk with center $z.$
\end{theorem}

We define spherical counterparts to $C(\Omega)$ in the following way:
$$
\Cd(\Omega)=\inf_{z\in\Omega}\delta_\Omega(z)\mu_\Omega(z)
\aand
\Ce(\Omega)=\inf_{z\in\Omega}\varepsilon_\Omega(z)\mu_\Omega(z).
$$

\begin{example}\label{ex:1}
We consider the disk $\D_R=\{z\in\C: |z|<R\}$ for $0<R<+\infty.$
It is immediate to see that $C(\D_R)=1/2.$
On the other hand, we compute
$\mu_{\D_R}(z)=R(1+|z|^2)/(R^2-|z|^2),~\varepsilon_{\D_R}(z)=\tau(|z|,R)
=(R-|z|)/(1+R|z|)$ and $\delta_{\D_R}(z)
=\sigma(|z|,R)=(R-|z|)/\sqrt{(1+R^2)(1+|z|^2)}.$
Therefore,
$$
\Cd(\D_R)
=\inf_{0<x<R}\frac{R-x}{\sqrt{(1+R^2)(1+x^2)}}\cdot\frac{R(1+x^2)}{R^2-x^2}
=\inf_{0<x<R}\frac{R\sqrt{1+x^2}}{(R+x)\sqrt{1+R^2}}.
$$
Since the function $\sqrt{1+x^2}/(R+x)$ is decreasing in $0<x<1/R$
and increasing in $1/R<x,$ we obtain
$$
\Cd(\D_R)=\begin{cases}
1/2 &\quad\text{if}~R\le 1, \\
R/(1+R^2)<1/2 &\quad\text{if}~R>1.
\end{cases}
$$
We also have
$$
\varepsilon_{\D_R}(x)\mu_{\D_R}(x)
=\frac{R-x}{1+Rx}\cdot\frac{R(1+x^2)}{R^2-x^2}
=\frac{R(1+x^2)}{(1+Rx)(R+x)}
$$
for $0<x<R.$
Since the function $R(1+x^2)/(1+Rx)(R+x)$ is decreasing in $0<x<1,$
increasing in $x>1,$ and tends to $1/2$ as $x\to R,$ we obtain finally
$$
\Ce(\D_R)=\begin{cases}
1/2 &\quad\text{if}~R\le 1, \\
2R/(1+R)^2<1/2 &\quad\text{if}~R>1.
\end{cases}
$$
The spherical diameter, namely, 
the diameter with respect to the distance $\sigma,$
of a set $E\subset\sphere$ will be denoted by $\sigma(E).$
Then we observe that
$$
\sigma(\sphere\setminus\D_R)=
\begin{cases}
1 &\quad\text{if}~R\le 1, \\
\sigma(R,-R)=2R/(1+R^2) &\quad\text{if}~R>1.
\end{cases}
$$
Therefore,
$$
\frac{\Cd(\D_R)}{\sigma(\sphere\setminus\D_R)}=\frac12
\aand
\frac12\le\frac{\Ce(\D_R)}{\sigma(\sphere\setminus\D_R)}<1
$$
for any $R>0.$
Note also that the diameter of $\sphere\setminus\D_R$ with
respect to $\tau$
is $+\infty$ for $R\le 1$ and $2R/(R^2-1)$ for $R>1.$
\end{example}

In view of the above example, we expect more uniform estimates
if we consider the modified quantities
$$
\Cdm(\Omega)=\frac{\Cd(\Omega)}{\sigma(\sphere\setminus\Omega)}
\aand
\Cem(\Omega)=\frac{\Ce(\Omega)}{\sigma(\sphere\setminus\Omega)}.
$$

Since $\delta_\Omega, \varepsilon_\Omega, \mu_\Omega, 
\sigma(\sphere\setminus\Omega)$ are invariant
under the spherical isometries (see \cite{Minda85a}),
so are the quantities $\Cd(\Omega), \Cdm(\Omega), \Ce(\Omega)$ and
$\Cem(\Omega);$ namely, $\Cd(T(\Omega))=\Cd(\Omega),
\Cdm(T(\Omega))=\Cdm(\Omega), \Ce(T(\Omega))=\Ce(\Omega)$
and $\Cem(T(\Omega))=\Cem(\Omega)$ for a spherical isometry $T.$

Our main result is now stated as in the following.

\begin{theorem}[Main Theorem]
Let $\Omega$ be a hyperbolic domain in $\C.$
Then,
\begin{enumerate}
\item[(i)]
$\Ce(\Omega)\le1/2.$
\item[(ii)]
$\Cd(\Omega)\le \Ce(\Omega)$ and $\Cdm(\Omega)\le \Cem(\Omega).$
\item[(iii)]
$\Ce(\Omega)\le 2C(\Omega).$
\item[(iv)]
$C(\Omega)\le 4\Cdm(\Omega)=4\Cd(\Omega)/\sigma(\sphere\setminus\Omega).$
\end{enumerate}
\end{theorem}

As an immediate corollary of the main theorem, we obtain the following
characterizations of uniform perfectness of the boundary.

\begin{corollary}
Let $\Omega$ be a hyperbolic domain in $\sphere.$
Then the following conditions are equivalent:
\begin{enumerate}
\item[(1)]
$\partial\Omega$ is uniformly perfect.
\item[(2)]
$\Cd(\Omega)>0.$
\item[(3)]
$\Ce(\Omega)>0.$
\end{enumerate}
\end{corollary}

Harmelin and Minda \cite{HM92} showed that $C(\Omega)\le 1/2$
for a hyperbolic domain $\Omega\subset\C.$
The above assertion (i) (and thus $\Cd(\Omega)\le 1/2$)
can be regarded as a spherical analog of it.
In addition, Mejia and Minda \cite{MM90} showed that $C(\Omega)\ge1/2$
if and only if $\Omega$ is convex.
Let us mention the following result due to Minda.

\begin{theorem}[Minda $\text{\cite[Theorem 1]{Minda86}}$]
Let $\Omega$ be a spherically convex domain in $\sphere$ and $z\in\Omega.$
Then
$$
\mu_\Omega(z)\ge\frac{1+\varepsilon_\Omega(z)^2}{2\varepsilon_\Omega(z)},
$$
where equality holds if and only if $\Omega$ is a hemisphere.
\end{theorem}

In particular, $\varepsilon_\Omega(z)\mu_\Omega(z)\ge
(1+\varepsilon_\Omega(z)^2)/2>1/2$ and hence,
$$
\Ce(\Omega)\ge1/2
$$
for a spherically convex domain $\Omega.$
This gives a spherical analog to the one direction of the afore-mentioned
result.
We observe that $\Cd(\D_R)=\Ce(\D_R)=1/2$ for $0<R\le 1$
and $\Cd(\D_R)<\Ce(\D_R)<1/2$ for $R\ge1$ in Example \ref{ex:1}.
Since $\D_R$ is spherically convex if and only if $0<R\le 1,$
we have some hope that the conditions
$\Cd(\Omega)\ge1/2$ and/or $\Ce(\Omega)\ge1/2$
would characterize spherical convexity of $\Omega.$

\section{Spherical geometry}

In this section, we collect necessary information about the spherical geometry
to prove our main theorem.

Let $\mob$ be the group of M\"obius transformations $z\mapsto (az+b)/(cz+d),$
with $a,b,c,d\in\C,~ad-bc\ne0.$
This is nothing but the group of analytic automorphisms of the Riemann
sphere (the complex projective line) and is canonically isomorphic to
$\PSL(2,\C)=\SL(2,\C)/\{\pm I\}.$
Note that the action of $\mob$ on $\sphere$ is not isometric with
respect to the spherical metric $\sigma=|dz|/(1+|z|^2).$
We denote by $\isom^+(\sphere)$ the subgroup of $\mob$
consisting of spherical isometries.
It is a standard fact that each isometry $T\in\isom^+(\sphere)$ has either the
form
$$
T(z)=e^{i\theta}\frac{z-a}{1+\bar az}
$$
for a real constant $\theta$ and a complex number $a\in\C,$
or the form $T(z)=-e^{i\theta}/z$ for a real constant $\theta,$ in which case
we can interpret $a=\infty.$
In particular, we can see that $\isom^+(\sphere)$ acts on $\sphere$
transitively.
Note that $\tau(z,a)=|T(z)|$ for the above $T.$
It is also useful to note the relations
$$
\varepsilon_{T(\Omega)}(T(z))\mu_{T(\Omega)}(T(z))
=\varepsilon_{\Omega}(z)\mu_{\Omega}(z)
$$
and
$$
\delta_{T(\Omega)}(T(z))\mu_{T(\Omega)}(T(z))
=\delta_{\Omega}(z)\mu_{\Omega}(z),
$$
in particular,
$$
\Ce(T(\Omega))=\Ce(\Omega)
\aand
\Cd(T(\Omega))=\Cd(\Omega)
$$
for $T\in\isom^+(\sphere).$
Likewise, we also have $\Cdm(T(\Omega))=\Cdm(\Omega).$

Recall that $0\le\sigma(z,w)\le1$ and that $z$ and $w$ are called {\it antipodal}
if $\sigma(z,w)=1,$ which is equivalent to $\tau(z,w)=+\infty.$
It is easy to see that $z$ and $w$ are antipodal if and only if $z=-1/\bar w.$
We write $z^*=-1/\bar z$ for the aintipodal point of $z.$
It should be noted here that $\delta_\Omega(z)<1$ holds always for
a hyperbolic domain $\Omega.$

We have a simple relation between $\sigma$ and $\tau.$
Since
\begin{align*}
1+\tau(z,w)^2
&=\frac{|1+z\bar w|^2+|z-w|^2}{|1+z\bar w|^2} \\
&=\frac{(1+|z|^2)(1+|w|^2)}{|1+z\bar w|^2} \\
&=\frac{\tau(z,w)^2}{\sigma(z,w)^2},
\end{align*}
we have
$$
\sigma(z,w)=\frac{\tau(z,w)}{\sqrt{1+\tau(z,w)^2}}
\aand
\tau(z,w)=\frac{\sigma(z,w)}{\sqrt{1-\sigma(z,w)^2}}.
$$
In particular, $\sigma(z,w)\le\tau(z,w).$
We also have the relation $\delta_\Omega(z)
=\varepsilon_\Omega(z)/\sqrt{1+\varepsilon_\Omega(z)^2}$ for
a hyperbolic domain $\Omega.$

We now compare $\varepsilon_\Omega(z)$ with $d_\Omega(z).$

\begin{lemma}\label{lem:1}
Let $\Omega$ be a hyperbolic domain in $\C$ and fix a point $z\in\Omega.$
Then, $\varepsilon_\Omega(z)|z|\le 1$ and
$$
\frac{\varepsilon_\Omega(z)(1+|z|^2)}{1+\varepsilon_\Omega(z)|z|}
\le d_\Omega(z)\le
\frac{\varepsilon_\Omega(z)(1+|z|^2)}{1-\varepsilon_\Omega(z)|z|}.
$$
\end{lemma}

\begin{proof}
For brevity, set $\varepsilon=\varepsilon_\Omega(z)$ and
let $\Delta=\{w\in\sphere: \tau(w,z)<\varepsilon\}.$
Then, by assumption, $\Delta\subset\Omega\subset\C.$
Let $T(w)=(z-w)/(1+\bar zw).$
Note that $T\inv=T.$
Then $\Delta=T\inv(\D_\varepsilon)=T(\D_\varepsilon).$
Since $\Delta$ does not contain $\infty,$ the function $T$ does
not have a pole in $\D_\varepsilon,$ which implies $\varepsilon|z|\le1.$
If $\varepsilon|z|=1,$ $\Delta$ is a half-plane and $T$ has a pole at
$z^*.$
Note that the image of the diameter $[z^*, -z^*]$ of $\D_\varepsilon$
under $T$ is a half-line perpendicular to $\partial\Delta.$
The Euclidean distance from $z$ to $\partial\Delta$ is thus
$$
|T(-z^*)-T(0)|=\frac{|z-z^*|}2=\frac{1+|z|^2}{2|z|}
=\frac{\varepsilon(1+|z|^2)}{1+\varepsilon|z|}.
$$
The assertion is now confirmed in this case.
We next assume that $\varepsilon|z|<1.$
We then compute
$$
\frac{|1+\bar zw|^2(\tau(w,z)^2-\varepsilon^2)}{1-\varepsilon^2|z|^2}
=\left|w-\frac{(1+\varepsilon^2)}{1-\varepsilon^2|z|^2}z\right|^2
-\left(\frac{\varepsilon(1+|z|^2)}{1-\varepsilon^2|z|^2}\right)^2,
$$
which means that $\Delta$ is the disk with center 
$m=(1+\varepsilon^2)z/(1-\varepsilon^2|z|^2)$ and radius
$r=\varepsilon(1+|z|^2)/(1-\varepsilon^2|z|^2).$
Since a point $a$ in $\partial\Delta$ belongs to $\partial\Omega,$ we have
$$
d_\Omega(z)\le |z-a|\le r+|z-m|=\frac{\varepsilon(1+|z|^2)}{1-\varepsilon|z|}.
$$
On the other hand, we obtain
$$
d_\Omega(z)\ge d_\Delta(z)=r-|z-m|=
\frac{\varepsilon(1+|z|^2)}{1+\varepsilon|z|}.
$$
Thus the proof is complete.
\end{proof}

\section{Proof of the main theorem}

Before the proof of the main theorem, we prepare a couple of lemmas which
will be used later.
We will call a map $f:\Omega\to\C$ {\it disk-convex} if $f$ maps any
disk in $\Omega$ conformally onto a convex domain.
Note that any M\"obius transformation $T$ is disk-convex on $\Omega$ whenever
$T(\Omega)\subset\C.$

\begin{lemma}
Suppose that $f$ maps a hyperbolic domain $\Omega$ in $\C$ conformally onto
another hyperbolic domain $\Omega'$ in $\C.$
If $f$ is disk-convex, then for each $z\in\Omega,$
$$
d_\Omega(z)|f'(z)|\le 2d_{\Omega'}(f(z)).
$$
\end{lemma}

\begin{proof}
Fix $z_0\in\Omega$ and set $d_0=d_\Omega(z_0).$
Since $f$ is convex on the disk $\Delta=\{z: |z-z_0|<d_0\},$
a covering theorem for convex functions (see \cite[Theorem 2.15]{Duren:univ})
implies that $f(\Delta)\supset\{w: |w-f(z_0)|<d_0|f'(z_0)|/2\}.$
Thus $d_{\Omega'}(f(z_0))\ge d_{f(\Delta)}(f(z_0))\ge d_0|f'(z_0)|/2.$
\end{proof}

Since $\lambda_{\Omega'}(f(z))|f'(z)|=\lambda_\Omega(z),$ we obtain the
following.

\begin{corollary}\label{cor:2}
For a hyperbolic domain $\Omega$ in $\C$ and a disk-convex univalent function 
$f:\Omega\to\C,$ 
$$
d_\Omega(z)\lambda_\Omega(z)\le
2d_{f(\Omega)}(f(z))\lambda_{f(\Omega)}(f(z)).
$$
In particular, $C(\Omega)\le 2C(f(\Omega)).$
\end{corollary}

\begin{remark}
In \cite{HM92}, Harmelin and Minda proved that $C(f(\Omega))\le A C(\Omega)$
for a conformal map $f$ with constant $A=\sqrt{1+3\coth^2(\pi/4)}=2.8241\dots$
and conjectured that $A$ can be reduced to $2.$
Later, Ma and Minda \cite{MM95} obtained a better bound:
$A=\sqrt{1+3\coth^2(\pi/3)}=2.4335\dots$
\end{remark}

\begin{proof}[Proof of the main theorem]
We first prove assertion (i).
The idea employed in the proof of Harmelin and Minda \cite[Theorem 4]{HM92}
works.
Fix a point $z_0\in\Omega$ and set $R=\varepsilon_\Omega(z_0).$
Take a boundary point $a\in\partial\Omega$ such that $R=\tau(z_0,a).$
By a suitable spherical isometry, we may assume that $z_0=0$ and $a>0$
(and hence, $a=R$).
Then, $\D_R\subset\Omega$ and thus $\mu_\Omega\le\mu_{\D_R}$ on $\D_R.$
Note also that $\varepsilon_\Omega(x)=\varepsilon_{\D_R}(x)=\sigma(x,R)$
for $0<x<R.$
Hence, by Example \ref{ex:1},
$$
\Ce(\Omega)\le
\lim_{x\to R^-}\varepsilon_\Omega(x)\mu_\Omega(x)
\le\lim_{x\to R^-}\varepsilon_{\D_R}(x)\mu_{\D_R}(x)\le\frac12.
$$

Assertion (ii) is obvious because $\delta_\Omega(z)\le\varepsilon_\Omega(z).$

We next show assertion (iii).
By definition and Lemma \ref{lem:1}, we observe
$$
d_\Omega(z)\lambda_\Omega(z)
\ge\frac{\varepsilon_\Omega(z)(1+|z|^2)}{1+\varepsilon_\Omega(z)|z|}
\cdot\frac{\mu_\Omega(z)}{1+|z|^2}
\ge\frac{\varepsilon_\Omega(z)\mu_\Omega(z)}2,
$$
from which the inequality $C(\Omega)\ge\Ce(\Omega)/2$ follows.

Finally, we show assertion (iv).
Fix a point $z\in\Omega$ and take a point $a\in\partial\Omega$
such that $\delta_\Omega(z)=\sigma(z,a).$
Then take a point $b\in\sphere\setminus\Omega$ so that
$\max_{w\in\sphere\setminus\Omega}\sigma(w,a)=\sigma(b,a).$
It is easy to see the inequality
$$
\frac12\sigma(\sphere\setminus\Omega)\le \sigma(a,b)\le 
\sigma(\sphere\setminus\Omega),
$$
where $\sigma(\sphere\setminus\Omega)$ is the spherical diameter of
$\sphere\setminus\Omega.$
Let $T\in\isom^+(\sphere)$ such that $T(b)=\infty.$
Then $a'=T(a)\ne\infty$ and $\sigma(a,b)=\sigma(a',\infty)=1/\sqrt{1+|a'|^2}.$
Set $\Omega'=T(\Omega)$ and $z'=T(z).$
Note here that $\delta_{\Omega'}(z')=\sigma(z',a').$
Then, by the above observations and Corollary \ref{cor:2}, we have
\begin{align*}
\delta_\Omega(z)\mu_\Omega(z)
&=\delta_{\Omega'}(z')\mu_{\Omega'}(z') \\
&=\frac{\sqrt{1+|z'|^2}}{\sqrt{1+|a'|^2}}|z'-a'|\lambda_{\Omega'}(z') \\
&\ge \frac{\sigma(\sphere\setminus\Omega)\sqrt{1+|z'|^2}}2d_{\Omega'}(z')
\lambda_{\Omega'}(z') \\
&\ge \frac{\sigma(\sphere\setminus\Omega)}2\cdot\frac{d_{\Omega}(z)
\lambda_{\Omega}(z)}2.
\end{align*}
Hence, we obtain the inequality $\Cdm(\Omega)\ge C(\Omega)/4.$
\end{proof}

\noindent
{\bf Acknowledgement.}
The author would like to express his sincere thanks to the referee for
careful reading and corrections.

\end{document}